\documentclass{amsart}
\usepackage{amssymb}

\newtheorem{theorem}{Theorem}

\newtheorem{example}[theorem]{Example}

\newtheorem{lemma}[theorem]{Lemma}

\def\eps{\varepsilon}

\begin{document}

\title[ACIM for non-autonomous dynamical systems]{Absolutely Continuous Invariant measures for non-autonomous dynamical systems.}
\thanks{The research of the authors was supported by NSERC grants. }
\date{\today }
\author[P. G\'ora]{Pawe\l\ G\'ora }
\address[P. G\'ora]{Department of Mathematics and Statistics, Concordia
University, 1455 de Maisonneuve Blvd. West, Montreal, Quebec H3G 1M8, Canada}
\email[P. G\'ora]{pawel.gora@concordia.ca}
\author[A. Boyarsky]{Abraham Boyarsky }
\address[A. Boyarsky]{Department of Mathematics and Statistics, Concordia
University, 1455 de Maisonneuve Blvd. West, Montreal, Quebec H3G 1M8, Canada}
\email[A. Boyarsky]{abraham.boyarsky@concordia.ca}
\author[Ch. Keefe]{Christopher Keefe}
\address[Ch. Keefe]{Department of Mathematics and Statistics, Concordia
University, 1455 de Maisonneuve Blvd. West, Montreal, Quebec H3G 1M8, Canada}
\email[Ch. Keefe]{chriskeefe3.14159@gmail.com}



\begin{abstract}
We consider the non autonomous dynamical system $\{\tau _{n}\},$ where $\tau
_{n}$ is a continuous map $X\rightarrow X,$ and $X$ is a compact metric
space. We assume that $\{\tau _{n}\}$ converges uniformly to $\tau .$ The
inheritance of chaotic properties as well as topological entropy by $\tau $
from the sequence $\{\tau _{n}\}$ has been studied in \cite{Can1, Can2, Li,Ste,Zhu}. In \cite{You} the
generalization of SRB\ measures to non-autonomous systems has been
considered. In this paper we study absolutely continuouus invariant measures
(acim) for non autonomous systems. After generalizing the Krylov-Bogoliubov 
Theorem  \cite{KB} and Straube's Theorem \cite{Str} to the non autonomous setting, we prove that
under certain conditions the limit map $\tau $ of a non autonomous sequence
of maps $\{\tau_n\}$ with acims has an acim.
\end{abstract}

\maketitle

\section{Introduction}

Autonomous systems are rare in nature. A more realistic approach to
modeling real life processes is to consider non autonomous models. In
this note we consider a sequence of maps $\{\tau _{n}\}$ on a compact metric space $%
X\rightarrow X$. We assume that $\{\tau _{n}\}$ converges uniformly to $\tau .$
Let $\tau_{(0,n)}=\tau_{n}\circ\tau_{n-2}\circ\dots\circ\tau_{1}\circ\tau_0.$
For an initial measure $\eta$ we consider the sequence $\mu_n=(\tau_{(0,n)})_* \eta$.
 Since $X$
is compact the space of probability measures on $X$ is $*$-weakly compact and hence we can assume that $\{\mu
_{n}\}$ converges to a measure $\mu .$ In this note we study conditions
under which the limit map $\tau $ preserves $\mu$. In particular we are interested in the situation when 
$\mu_n$ and $\mu$ are absolutely continuous.

The behaviour of non autonomous sequences of piecewise expanding maps was studied before.
In the paper \cite{OSY} the authors consider a family $\mathcal E$ of exact piecewise expanding maps
with uniform expanding properties and show that for any two initial densities $f_1$, $f_2$
the iterates $P_{\tau_{(0,n)}}f_1$ and $P_{\tau_{(0,n)}}f_2$ get closer to each other with exponential speed.
Using the notation of Section 2:
$$\int |P_{\tau_{(0,n)}}f_1-P_{\tau_{(0,n)}}f_2| dm \le  C(f_1,f_2) \Lambda^n,\ \ n\ge 1,$$
for some constants $C(f_1,f_2)>0$, $0< \Lambda<1$ and any sequence of maps $\tau_n\in\mathcal E$.
In this situation, in general, there is no limit map and the densities $P_{\tau_{(0,n)}}f$ do not converge.
In this note we assume the uniform convergence $\tau_n \rightrightarrows \tau $. This allows us to prove that, under some assumptions,
the densities $P_{\tau_{(0,n)}}f$ converge to a $\tau$-invariant density.

Another approach to dealing with compositions of different maps is to consider a random map. Maps from a family 
$\mathcal E=\{\tau_a\}_{a\in\mathcal A}$ are applied randomly according to a probability on $\mathcal A$, which might depend on the current position of the process.
The literature on random maps is quite rich. A recent article is \cite{BBR}. The authors study, in particular, random maps based on the set $\mathcal E$ of the Liverani-Saussol-Vaienti maps
$$\tau_a(x)=\begin{cases} x(1+2^ax^a),&\ x\in [0,1/2],\\
                          2x-1,& \ x\in (1/2,1],
\end{cases}
$$
 with parameters in $ [a_0,a_1]\subset (0, 1)$ chosen independently with
respect to a distribution $\nu$ on $ [a_0,a_1]$. These maps have indifferent fixed points which makes them non-exponentially mixing. The authors study the fibre-wise (quenched) dynamics of the system. For this point of view a skew-product approach is convenient. 

Let $(\mathcal A,\mathcal  F,  p)$  be a Borel probability space, let 
$\Omega = \mathcal A^\mathbb Z$ be equipped with the product
measure $ P := p^\mathbb Z$ and let $\sigma:\Omega\to\Omega$
 denote the $P$-preserving two-sided shift map.
Let $ (X,\mathcal B)$  be a measurable space. Suppose that $\tau_a : X\to  X$  is a family of measurable
maps defined for $ p$-almost every $ a \in\mathcal  A$ such that the skew product
$$ T : X \times \Omega\to  X \times \Omega,\  T(x,\omega)=(\tau_{[\omega]_0},\sigma\omega) ,$$
 is measurable with respect to $\mathcal B\times\mathcal F$.
If $X_\omega=X\times\{\omega\}$ denotes the fiber over $\omega$  and
 $$\tau^n_\omega=\tau_{\sigma^{n-1}\omega}\circ\dots\circ \tau_\omega:X_\omega\to X_{\sigma^n \omega},$$
 we have
$T^n(x,\omega)=(\tau^n_\omega(x), \sigma^n \omega$.
If a probability measure $\mu$ is $T$-invariant and $\pi_*\mu=P$ ($\pi$ is the projection onto $\Omega$), then there exists
a family of probability fiber measures $\mu_\omega$ on $X_\omega$ such that $\mu(A)=\int \mu_\omega(A) dP(\omega)$
for any $A\in \mathcal B\times\mathcal  F$. Since $\mu$ is $T$-invariant the measures $\{\mu_\omega\}$ form an equivariant family, i.e., $(\tau_\omega)_* \mu_\omega=\mu_{\sigma\omega}$ for almost all $\omega$.

The authors study future and past quenched correlations: given
$\phi,\psi: X\times \Omega\to \mathbb R$ the future and past fibre-wise correlations are defined as
$$Cor_{n,\omega}^{(f)}=\int(\phi\circ \tau^n_\omega)\psi d\mu_\omega- \int\phi d\mu_{\sigma^n\omega}\int\psi d\mu_{\omega},$$
$$Cor_{n,\omega}^{(p)}=\int(\phi\circ \tau^n_{\sigma^{-n}\omega})\psi d\mu_{\sigma^{-n}\omega}- \int\phi d\mu_{\omega}\int\psi d\mu_{\sigma^{-n}\omega}.$$

They prove that for  the random map based on family $\mathcal E$ there exists an equivariant family of measures $\mu_\omega$ which are absolutely continuous $P$-a.e., characterize their densities and show that both
future and past quenched correlations are of order $\mathcal O(n^{1-1/a_0}+\delta)$ for bounded $\phi$ and H\"older 
continuous $\psi$ and arbitrary $\delta>0$.
The system $(T,\mu)$ is mixing.

In this note we assume that $\tau_n\rightrightarrows \tau$ and consider the compositions 
$\tau_{(0,n)}=\tau_{n}\circ\tau_{n-2}\circ\dots\circ\tau_{1}\circ\tau_0$, so we can say that we study one fixed fiber
under very special assumptions.

In Section 2 we give the definitions and introduce the notation. In Section 3 we generalize the
Krylov-Bogoliubov Theorem \cite{KB} and Straube's Theorem \cite{Str} to the non
autonomous setting.  Section 4 is independent of the previous section.
We make stronger   assumptions on the $\tau_n$'s and establish the
existence of an acim for the limit map $\tau $ and show that any convergent subsequence of
$\{P_{\tau_{(0,n)}}f\}_{n\ge 1}$ converges to an invariant density of the limit map, where
$P_{\tau_{(0,n)}}$ is the Frobenius-Perron operator induced by $\tau_{(0,n)}$ and $f$ is a density.

\section{Notation and Definitions}

Let $(X,\rho)$ be a compact metric space.
Let $\{\tau_n\}$ be a sequence of maps $\tau_n:X\to X$ which  
converges uniformly to a continuous map $\tau$.
We shall consider the non-autonomous dynamical system defined by
$$x_{m+1}=\tau_m(x_m), \ \ \ m=0,1,2,\dots
$$
where we assume that $\tau_0$ is the identity and $x_0\in I$.

We write $$\tau_{(m,n)}=\tau_{n}\circ\tau_{n-2}\circ\dots\circ\tau_{m+1}\circ\tau_m, \ \ \ n>m .$$
In particular,
 $$\tau_{(0,n)}=\tau_{n}\circ\tau_{n-2}\circ\dots\circ\tau_{1}\circ\tau_0.$$

Let $\mathcal B(X)$ be the  $\sigma$-algebra of Borel subsets of $X$.

For a map $\tau:X\to X$ we define an operator on measures on $\mathcal B(X)$:
$$\tau_* \mu (A)=\mu(\tau^{-1}A),$$ for any measurable set $A$.

\section{Generalization of the Krylov-Bogoliubov Theorem and Straube's Theorem}

We will now prove a generalization of the Krylov-Bogoliubov Theorem:

\begin{theorem}\label{KBTh} Let $\{\tau_n\}$ be a sequence of transformations defining a nonautonomous
dynamical system on the metric compact space $X$ with a continuous limit $\tau$. We assume that the $\tau_n$'s converge uniformly to $\tau$.
Let $\eta$ be a fixed probability measure on $X$. Define the measures $\mu_n=\frac 1n \sum _{i=1}^n \nu_i$, where $\nu_i=\left(\tau_{(0,i)}\right)_*(\eta)$.  Let $\mu$ be a $*$-weak limit point of the sequence $\{\mu_n\}_{n\ge 1}$.
 Then $\mu$ is  a $\tau$-invariant measure, i.e., $\tau_*\mu=\mu$.
\end{theorem}

\begin{proof}
 We follow the proof of the original Krylov-Bogoliubov Theorem. Let $\eta$ be a
probability measure  $X$. Then the sequence
$\mu_n=\frac 1n \sum _{i=1}^n \nu_i$, where $\nu_i=\left(\tau_{(0,i)}\right)_*(\eta)$
is a sequence of probability measures and contains
 a convergent subsequence $\mu_{n_k}$. Let $\mu=\lim_{k\to\infty}\mu_{n_k}$.
We will prove that $\tau_*\mu=\mu$. To this end it is enough to show that for any $g \in C^0(X)$, 
$\mu(g)=\tau_*\mu(g)=\mu(g\circ \tau)$. 

We  estimate the difference 
\begin{equation}\begin{split}
&\left|\mu_n(g)-\mu_n(g\circ \tau)\right|= \frac 1n \left|\sum _{i=1}^n \nu_i(g)-\sum _{i=1}^n\nu_i(g\circ\tau)\right|\\
&=\frac 1n \left|\eta(g\circ\tau_{(0,1)})+\eta(g\circ\tau_{(0,2)})+\dots+\eta(g\circ\tau_{(0,n-1)})+\eta(g\circ\tau_{(0,n)})\right.\\
&\hskip 2 cm \left.- \eta(g\circ \tau\circ\tau_{(0,1)})-\eta(g\circ \tau\circ\tau_{(0,2)})-\dots-\eta(g\circ \tau\circ\tau_{(0,n-1)})-\eta(g\circ \tau\circ\tau_{(0,n)})\right|\\
&= \frac 1n \left|\eta(g\circ\tau_{(0,1)})+\sum _{i=2}^{n}\left(\eta(g\circ\tau_{(0,i)})-\eta(g\circ \tau\circ\tau_{(0,i-1)})\right)-\eta(g\circ \tau\circ\tau_{(0,n)})\right|.
\end{split}\end{equation}
Let $\omega_g$ be the modulus of continuity of $g$, i.e.,
 $$\omega_g(\delta)=\sup_{\rho(x,y)<\delta}|g(x)-g(y)|.$$
For an arbitrary $\eps>0$ we can find a $\delta>0$ such that $\omega_g(\delta)<\eps$. Since $\tau_n\to\tau$ uniformly for this $\delta$ we can find an $N\ge 1$ such that $\sup_{x\in X}\rho(\tau_n(x),\tau(x)) < \delta$ for all $n> N$.

For  $i> N$, we have 
\begin{equation*}\begin{split}
&\left|\eta(g\circ\tau_{(0,i)})-\eta(g\circ \tau\circ\tau_{(0,i-1)})\right|
=
\left|\eta(g\circ\tau_i\circ\tau_{(0,i-1)}-g\circ \tau\circ\tau_{(0,i-1)})\right|\\
&=
\left|\eta((g\circ\tau_i-g\circ \tau)(\tau_{(0,i-1)}))\right|\le\omega_g(\delta)<\eps.
\end{split}\end{equation*}
Thus, for $n> N$, we have
$$\left|\mu_n(g)-\mu_n(g\circ \tau)\right|\le \frac 1n\left( N\cdot 2\cdot \sup |g|+ (n-N)\eps\right),$$
which becomes arbitrarily close to $\eps$ as $n\to\infty$. This shows that \newline $\mu_{n_k}(g)-\mu_{n_k}(g\circ \tau)\to 0$
as $k\to\infty$.

We have $\mu_{n_k}(g)\to\mu(g)$ and since $\tau$ is continuous $\mu_{n_k}(g\circ \tau)\to \mu(g\circ \tau)=\tau_*\mu(g)$. Thus, $\mu$ is a $\tau$-invariant measure.
\end{proof}

\textbf{Remark:} The only place where we needed the continuity of $\tau$ is the last line of the proof:
since $\tau$ is continuous $g\circ \tau$ is continuous for any continuous $g$ and then the $*$-weak convergence 
of $\mu_{n_k}$ implies $\mu_{n_k}(g\circ \tau)\to \mu(g\circ \tau)$.

Theorem \ref{KBTh} does not yield any more information about the $\tau$-invariant measure $\mu$. The next
result is a generalization of a theorem by Straube \cite{Str}, which provides a sufficient condition for $\mu$ to be absolutely continuous.

\begin{theorem}\label{STTh} Let $(X,\mathcal B, \nu)$ be a normalized measure space and let $\{\tau_n\}$ be
 a sequence of non-singular transformations defining a non-autonomous dynamical system on $X$. We do not
 assume that the limit $\tau$ is continuous.
 Assume there exists  $\delta>0$ and $0<\alpha<1$ such that
$$\nu(E)<\delta \ \Longrightarrow\ \sup_{k\ge 1}\ \nu\left(\tau^{-1}_{(0,k)}(E)\right)<\alpha,$$
for all $E\in\mathcal B$. Then there exists a $\tau$-invariant normalized measure $\mu$ which is absolutely
continuous with respect to $\nu$.
\end{theorem}

(The proof uses a number of facts from the theory of finitely additive measures which are collected in the Appendix.
The proof is similar to the proof in \cite{Str} but is modified to allow the use of the estimates from the proof of Theorem \ref{KBTh}.)

\begin{proof}
Let us define the measures $$\nu_n(E)=\frac 1n \sum_{k=0}^{n-1}\nu(\tau^{-1}_{(0,k)}(E))\ ,\ E\in \mathcal B.$$
Then, for all $n$,

(a) $ \nu_n(X)=1$;

(b) $\nu_n\ll \nu$ \ ($\tau_n$ is non-singular for every $n$);

(c) $\nu_n(\cdot)\ge 0$.

Thus, $\{\nu_n\}$ is a sequence of positive, normalized, absolutely continuous measures and can be treated
as a sequence in the unit ball of $L^*_\infty(X)$ with the $*$-weak topology. Thus, it contains a convergent 
subsequence $\nu_{n_k}\to z$ and $z$ can be identified with a finitely additive measure on $X$.
The measure $z$ is finitely additive, positive, normalized and absolutely continuous with respect to $\nu$.

By Lemma  \ref{Lem1} in the Appendix we can uniquely decompose $z$ into
$$z = z_c + z_p,$$
where $z_c$ is countably additive and $z_p$ is purely finitely additive. Now, we claim that $z_c \not= 0$.  Otherwise, by Lemma \ref{Lem0}, there exists a
decreasing sequence $\{E_n\}\subset\mathcal B$ such that $\lim_{n\to\infty}\nu(E_n) = 0$
and $z(E_n) = z(X)=1$ for all $n\ge 1$.
Since $\nu(E_n)\to 0$, for any $\delta > 0$, there exists an $n_0$ such that $n > n_0 \Longrightarrow \nu(E_n) < \delta$. 
Now, by our assumptions, there is an $\alpha < 1$ such that,
$$\sup_k \nu(\tau^{-1}_{(0,k)}(E_n))<\alpha< 1.$$
Thus, 
$\nu(\tau^{-1}_{(0,k)}(E_n)<\alpha $ for all $k$. So,
$$z(E_n) < \alpha < 1,$$
which is a contradiction. We have demonstrated that $z_c \not= 0$.

 Now we will prove that $z_c$ is $\tau$-invariant.
Consider the finitely additive measure 
$$\kappa=z-z\circ\tau^{-1}=z_c-z_c\circ\tau^{-1}+z_p-z_p\circ\tau^{-1}.$$
In the proof of   Theorem \ref{KBTh} we showed that for any continuous function
$g$ on $X$ we have
 $$ \mu_{n_k}(g)-\mu_{n_k}(\tau^{-1}(g))\to 0\ ,\ k\to\infty.$$
This means that for any  continuous function
$g$  (which is bounded since $X$ is compact) we have $$\kappa(g)=z(g)-z\circ\tau^{-1}(g)=0.$$
We do not need continuity of $\tau$ here as $ \mu_{n_k}(h)\to z(h)$ for all bounded $h$.
By Lemma \ref{Lem3} in the Appendix the countably additive component of $\kappa$ is 0, which means
$$z_c-z_c\circ\tau^{-1}=0,$$
or that $z_c$ is $\tau$-invariant.
\end{proof}


In the following example we show that, unlike in the case of one transformation, the converse implication in Theorem \ref{STTh} may not hold. 
We will construct a  sequence of maps $\tau_n\to\tau$, such that $\tau$ admits an acim and  
\begin{equation}\label{cond}\forall_{ \ \delta>0}\ \exists _{\ E\in\mathcal B}\ \ \sup_{k\ge 1}\ \nu\left(\tau^{-1}_{(2,k)}(E)\right)=1.\end{equation}

\begin{example} Let us consider maps $\tau_n:[0,1]\to [0,1]$, $n=2,3,\dots$, defined as follows
$$\tau_n(x)=\begin{cases} (1-\frac 1n)x, &\ \text{for}\ x\in [0,\frac 12);\\
                           2x-1,  &\ \text{for}\ x\in [\frac 12,1].
\end{cases}
$$
The limit map $\tau(x)=x\chi_{[0,\frac 12)}(x)+(2x+1)\chi_{ [\frac 12,1]}(x)$ admits an acim and condition (\ref{cond}) holds.
\end{example}
\def\rh{\varrho}
\begin{proof} Let $\rho_n={\tau_n}_{|[0,\frac 12)}$ be the first branch of $\tau_n$. The slope of $\rho_n=\frac{n-1}n$ so the slope
of $\rho_{m,n}=\rho_n\circ\rho_{n-1}\circ\rho_{n-2}\circ\dots\circ\rho_m$, $n>m$, is $\frac{n-1}{n}\cdot\frac{n-2}{n-1}\cdot\frac{n-3}{n-2}\cdot\dots\cdot\frac{m-1}{m}=\frac mn<1$.
Then, the interval $\rho_{m,n}^{-1}([0,\delta])$ is the interval from 0 to the minimum of $\delta\cdot\frac nm$ and $\frac 12$.
Note, that for any $k$, we have \begin{equation}\label{rho}\rho_k^{-1}([0,\frac 12])=[0,\frac 12].\end{equation}

Letting $\rh=\rh_n={\tau_n}_{|[\frac 12,1]}$ be the second branch of $\tau_n$, we have
\begin{equation*}\begin{split} \rh^{-1}\left(\left[0,\frac 12\right]\right)&=\left[\frac 12,\frac 12+\frac 14\right];\\
\rh^{-1}\left(\left[\frac 12,\frac 12+\frac 14\right]\right)&=\left[\frac 12+\frac 14,\frac 12+\frac 14+\frac 18\right];\\
&\vdots\\
\rh^{-1}\left(\left[\sum_{i=1}^k \frac 1{2^i},\sum_{i=1}^{k+1} \frac 1{2^i} \right]\right)&=\left[\sum_{i=1}^{k+1} \frac 1{2^i},\sum_{i=1}^{k+2} \frac 1{2^i}\right].
\end{split}
\end{equation*}
This and (\ref{rho}) imply that $$\tau_{(2,m-1)}^{-1}([0,\frac 12])=\left[0,\sum_{i=1}^{m-1} \frac 1{2^i} \right].$$

Let $\eps>0$ and $m$ such that $1-\sum_{i=1}^{m-1} \frac 1{2^i}<\eps$. Let $n$ satisfy $\delta\cdot\frac nm>\frac 12$. Then the Lebesgue measure of $\tau_{2,n}^{-1}([0,\delta])$ is larger than
$1-\eps$.
\end{proof}


\section{Existence of an absolutely continuous invariant measure for the limit map}

In this section we will assume that all the maps $\tau_n$ are piecewise expanding maps of an interval.
For  the general theory of such maps we refer the reader  to \cite{BG} or \cite{LM}.

Let $I=[0,1]$. The map $\tau:I\to I$ is called piecewise expanding iff
 there exists a partition ${\mathcal{P}}
=\{I_{i}:=[a_{i-1},a_{i}],i=1,\dots ,q\}$ of $I$ such that ${\tau }:{I}%
\rightarrow {I}$ satisfies the following conditions:

 (i) $\tau $ is monotonic on each interval $I_{i}$;

 (ii) $\tau _{i}:=\tau |_{I_{i}}$ is $C^{2}$, i.e., $C^{2}$ in the interior and the one-sided limits of the derivatives are finite at endpoints;

 (iii) $|\tau _{i}^{\prime }(x)|\geq s_i\ge s >1$ for any $i$ and for all $x\in
(a_{i-1},a_{i})$.

The following Frobenius-Perron operator $P_\tau:L^1(I,m)\to L^1(I,m)$, where $m$ is Lebesgue measure, is a basic tool in the theory of piecewise expanding maps.
For a general non-singular map $\tau$ $\left[m(A)=0\ \Longrightarrow\ m(\tau^{-1}(A)=0\right]$, we define $P_\tau f$ as a Radon-Nikodym derivative $\frac {d(\tau_*m)}{dm}$.
 For piecewise expanding maps the operator can be written explicitly \cite{BG}:
$$P_\tau f (x)=\sum_{i=1}^q\frac{f(\tau_i^{-1}(x))}{|\tau'(\tau_i^{-1}(x))|}.$$
In particular $P_\tau f=f$ iff $f\cdot m$ is an acim of $\tau$.
Piecewise expanding maps of the interval satisfy the following Lasota-Yorke inequality \cite{LY}.
For any bounded variation function $f\in BV(I)$ the variation $V(P_\tau f)$ satisfies
$$V(P_\tau f)\le A V(f) + B\int_I |f| dm ,$$
where the constants $A=\frac 2 s$, $B=\frac {\max |\tau''|}{s}+\frac 2 h$ and $h=\min_i \{m(I_i)\}$. In particular, we can assume that $A<1$, considering an iterate $\tau^k$, if necessary.
We always assume that bounded variation functions are modified to satisfy $f(x_0)=\limsup_{x\to x_0} f(x)$ for all $x_0\in I$. 

We will prove the following:

\begin{theorem}\label{Th:acim} Assume that  $\tau_n$, $n=1,2,\dots$ are piecewise expanding maps of an interval
 and  satisfy the Lasota-Yorke inequality with common constants
$A<1$ and $B$. Then, for any density $f\in BV(I)$, the sequence $f_n=\frac 1n \sum _{i=1}^{n} P_{\tau_{(1,i)}}f$ 
forms a precompact set in $L^1$ and  any convergent subsequence converges to a density of an acim of the limit map $\tau$.
\end{theorem}

\textbf{Remark:} We do not assume that the maps $\tau_n$ are defined on a common partition. We assume that they all satisfy Lasota-Yorke inequality with the same constant $B$. In the following lemma we  show that this implies that the limit map $\tau$ is defined on a finite partition and the partitions for maps $\tau_n$ are ``asymptotically"
the same as the partition for $\tau$.

\begin{lemma}\label{Le:acim} Under the assumptions of Theorem \ref{Th:acim} the limit map $\tau$ is piecewise monotonic and there exists a constant
  $K$ such that for any interval $J$ we have $m(\tau^{-1}(J))\le K m(J)$. In particular, it follows that the limit map $\tau$ is non-singular.
\end{lemma}
\begin{proof} Since the constant $B$ depends on the reciprocal of $h$, there is a universal bound $q_u$ on the number of elements of the partition $\mathcal P$ for 
$\tau_n$. This places a restriction on the number $k$ of iterates we can use to make $A<1$. Thus, there exists a universal lower bound $s_u$ for the modulus of the derivative 
$\tau_n'$.

Now, we prove that $\tau$ is piecewise monotonic. Assume that the graph of $\tau$ contains $p$ points forming a ``zigzag", i.e., there exist  $x_1<x_2<x_3<\dots<x_{p-1}<x_p$
such that $\tau(x_i)<\tau(x_{i+1})$ for odd $i$ and $\tau(x_i)>\tau(x_{i+1})$ for even $i$ (or other way around). Then, $p\le 2q_u$. If not, then since $\tau_n\rightrightarrows \tau$ uniformly,
for large $n$ the graph of $\tau_n$ also contains a zigzag of length $p$. This is impossible as $\tau_n$ has at most $q_u$ branches of monotonicity.
Thus, $\tau$ is piecewise monotonic with at most $q_u$ branches of monotonicity.

Let $[a,b]\subset I$ be an interval. Each line $y=a$, $y=b$ intersects the graph of $\tau$ in at most $q_u$ points. Let points $(x_1,a)$, $(x_2,b)$ be the points of intersection
of these lines with one monotonic, say increasing, branch of $\tau$. Then,
$$b-a=\lim_{n\to\infty} \tau_n(x_2)-\tau_n(x_1)\ge \lim_{n\to\infty} s_u\cdot (x_2-x_1)=s_u\cdot(x_2-x_1).$$ If one (or two) of  the intersections is empty, we replace appropriate $x_i$ by the endpoint of 
the interval of monotonicity.  Thus, for any interval $J$ we have 
\begin{equation}\label{preimage_est}
m(\tau^{-1}(J))\le \frac{q_u}{s_u} m(J). 
\end{equation}
\end{proof}

We can now prove Theorem \ref{Th:acim}.
\begin{proof} [Proof of Theorem \ref{Th:acim}] Since $f$ is a density and the Frobenius-Perron operator preserves the integral of positive functions, we have $\int|P_{\tau_n}f|dm =1$ for all $n\ge 1$.
 Since $P_{\tau_{(1,i)}}=  P_{\tau_i}\circ P_{\tau_{i-1}}\circ\dots\circ P_{\tau_{2}}\circ P_{\tau_{1}}$, we can apply the Lasota-Yorke inequality consecutively and obtain
$$V(P_{\tau_{(1,i)}}f)\le A^i V(f)+B(A^{i-1}+A^{i-2}+\dots +A^2+A+1)\le A^iV(f)+\frac B{1-A}\ ,\ i\ge 1 .$$
Thus, the functions $P_{\tau_{(1,i)}}f$ and also the functions $f_n$, $i,n\ge 1$, have uniformly bounded variation. Since for a bounded variation density $f$, $\sup_{x\in I} f(x)\le 1+V(f)$,
these functions are also uniformly bounded. The sequence $\{f_n\}_{n\ge 1}$, being both uniformly bounded 
 and of uniformly bounded variation contains a  subsequence  $\{f_{n_k}\}_{k\ge 1}$ convergent almost everywhere to a function $f^*$ of bounded variation by Helly's Theorem \cite{Nat}.
Additionally, by the Lebesgue Dominated Convergence Theorem, $\int _I f^* dm =1$.
This means that, by Scheffe's Theorem \cite{Bil},  $f_{n_k}\to f^*$  in the $L^1$-norm.
Thus, the sequence $\{f_n\}_{n\ge 1}$ forms a pre-compact set in $L^1$ and in particular, contains a subsequence convergent in $L^1$ to a function of bounded variation.

Now, we will prove that  for any density $F$, $(P_{\tau_n}F-P_\tau F)\to 0$ weakly in $L^1$, as $n\to\infty$.
Let $g\in L^\infty(I,m)$ be an arbitrary bounded function and let us fix an $\eps>0$. By Lusin's Theorem \cite[Th. 7.10]{Fol} for any $\eta>0$ there exists an open set
$U\subset I$, $m(U)<\eta$, and a continuous function $G \in C^0(I)$ such that $g=G$ on $I\setminus U$ and $\sup |G|\le \|g\|_\infty$.
The Frobenius-Perron operator is a conjugate of the Koopman operator, that is for any $f\in L^1$ and any $g\in L^\infty$, we have $\int_I P_\tau f\cdot g \, dm=\int_I  f\cdot g\circ \tau\,  dm$.
Therefore, we can write
\begin{equation*}\begin{split}
&\left|\int_I\left( P_\tau F\cdot g-P_{\tau_n} F\cdot g\right)\, dm\right|\le\int_I F\left| g\circ \tau-g\circ{\tau_n} \right|\, dm\\
&=\int_I F\left| g\circ \tau-G\circ \tau+G\circ \tau-G\circ \tau_n+G\circ \tau_n-g\circ{\tau_n} \right|\, dm\\
&\le \int_{\tau^{-1}(U)} F\left| g\circ \tau-G\circ \tau\right|\, dm+\int_I F\left| G\circ \tau_n+G\circ \tau_n\right|\, dm+\int_{\tau_n^{-1}(U)} F\left| g\circ \tau_n-G\circ \tau_n\right|\, dm .
\end{split}
\end{equation*}
Let $\sup G\le \|g\|_\infty =M_g$. Let $I_F(t)=\sup_{\{A: m(A)< t\}} \int_A |F| \, dm$. It is known that $I_F(t)\to 0$ as $t\to 0$.
Let $\omega_G$ be the modulus of continuity of $G$: $\omega_G(t)=\sup_{|x-y|\le t}|G(x)-G(y)|$. Again, $\omega_G(t)\to 0$ as $t\to 0$.
Using estimate (\ref{preimage_est}) we  obtain
\begin{equation}\label{estimate P}\begin{split}
&\left|\int_I\left( P_\tau F\cdot g-P_{\tau_n} F\cdot g\right)\, dm\right|\\
&\le 2M_g I_F\left(\frac{q_u}{s_u}\eta\right)+\omega_G(\sup|\tau_n-\tau|)+ 2M_g I_F\left(\frac{q_u}{s_u}\eta\right)\\
&=\omega_G(\|\tau_n-\tau\|_\infty)+4M_g I_F\left(\frac{q_u}{s_u}\eta\right).
\end{split}
\end{equation}
Let us fix an $\eps>0$.
Since $\|\tau_n-\tau\|_\infty\to 0$, as $n\to\infty$ we can find $N\ge 1$ such that for all $n\ge N$ we have $\omega_G(\|\tau_n-\tau\|_\infty)<\eps$. We can also find an $\eta>0$ such that
$4M_g I_F\left(\frac{q_u}{s_u}\eta\right)<\eps.$
 This shows that
$(P_{\tau_n}F-P_\tau F)\to 0$ weakly in $L^1$, as $n\to\infty$. Note, that this convergence is uniform over precompact subsets of $L^1$, since the estimate
(\ref{estimate P}) can be made common for all $F$ in  such a set (the functions in a precompact set are uniformly integrable).

Let $\{f_{n_k}\}_{k\ge 1}$ be a subsequence of $\{f_{n}\}_{n\ge 1}$ convergent in $L^1$ to $f^*$. To simplify the notation we will skip the subindex $k$.
We will show that $f^*$ is the density of an acim of $\tau$, i.e., $P_\tau f^*=f^*$. We have
\begin{equation*}\begin{split}
P_\tau f^*=P_\tau(\lim_{n\to \infty} f_n)=\lim_{n\to \infty}P_\tau f_n.
\end{split}
\end{equation*}
We will show that $P_\tau f_n-f_n$ converges weakly in $L^1$ to 0.
Let $\phi_i=P_{\tau_{(1,i)}} f$, $i=1,2,\dots$. Then, $f_n=\frac 1n\left(\phi_1+\phi_2+\dots+\phi_{n-1}+\phi_n\right)$. We can write
\begin{equation*}\begin{split}
&P_\tau f_n-f_n=\frac 1n\left(P_\tau\phi_1+P_\tau\phi_2+P_\tau\dots+P_\tau\phi_{n-1}+P_\tau\phi_n\right)-
\frac 1n\left(\phi_1+\phi_2+\dots+\phi_{n-1}+\phi_n\right)\\
&=\frac 1n\left(P_\tau\phi_n-\phi_1\right)+\frac 1n \sum_{i=1}^{n-1}\left(P_\tau \phi_i-\phi_{i+1}\right)
=\frac 1n\left(P_\tau\phi_n-\phi_1\right)+\frac 1n \sum_{i=1}^{n-1}\left(P_\tau \phi_i-P_{\tau_{i+1}}\phi_{i}\right).
\end{split}
\end{equation*}
 Let $I_\Phi$ be a common $I_F$ function for all $\phi_i$'s. Let $N$ and $\eta$ be chosen as above. Let $n\ge N+2$. Then, using  estimate (\ref{estimate P}), we have
\begin{equation*}\begin{split}
&\left|\int_I \left(P_\tau f_n-f_n\right) g\, dm\right|\\
&\le
\frac 1n\int_I\left|\left(P_\tau\phi_n-\phi_1\right) g\right|\, dm +\frac 1n \sum_{i=1}^{N}\int_I\left|\left(P_\tau \phi_i-P_{\tau_{i+1}}\phi_{i}\right)g\right|\, dm\\
&\hskip 4 cm +\frac 1n \sum_{i=N+1}^{n-1}\int_I\left|\left(P_\tau \phi_i-P_{\tau_{i+1}}\phi_{i}\right)g\right|\, dm\\
&\le \frac 2n M_g +\frac 2n N M_g+\frac {n-1-N}n \left(2\eps\right).
\end{split}
\end{equation*}
As $n\to \infty$ the right hand side becomes smaller than say $3\eps$. Since $\eps>0$ is arbitrary this proves that
$P_\tau f_n-f_n$ converges weakly in $L^1$ to 0 and $P_\tau f^*=f^*$.
\end{proof}

\section{Appendix} Here we collect the results about finitely additive measures necessary for the proof of Theorem
\ref{STTh}

\begin{lemma}\label{Lem0}{\rm [Theorem 1.22  of \cite{Yos}]}
Let $(X,\mathcal B)$ be a compact measure space. Let the measure $\eta$ be purely finitely
additive and $\eta\ge 0$. Let $\kappa$ be a countably additive measure defined on $(X,\mathcal B)$ such that
$\kappa\ge 0$. Then, there exists a decreasing sequence $\{E_n\}\subset\mathcal B$ such that $\lim_{n\to\infty}\kappa(E_n) = 0$
and $\eta(E_n) = \eta(X)$ for all $n\ge 1$. Conversely, if $kappa$ is a measure and the above conditions hold for
all countably additive $\kappa$, then $\eta$ is purely finitely additive.
\end{lemma}

\begin{lemma} \label{Lem1}{\rm [Theorems 1.23 and 1.24 of \cite{Yos}]} Let $\eta$ be a measure such that $\eta\ge 0$. Then there exist unique measures $\eta_p$ and
$\eta_c$ such that $\eta_p\ge 0$,
$\eta_c\ge 0$, $\eta_p$  is purely finitely additive, $\eta_c$ is countably additive and
$$\eta=\eta_p+\eta_c.$$
\end{lemma}

\begin{lemma}\label{Lem2}{\rm [Contained in the proof of Theorem 1.23  of \cite{Yos}]}
Let $\eta$ be a measure decomposed as $\eta=\eta_p+\eta_c.$. Then, $\eta_c$ is the greatest of
the measures $\kappa$, such that $0 \le \kappa\le\eta$.
\end{lemma}

\begin{lemma} \label{Lem3} If $\eta$ is a non-negative finitely additive measure and $$\int_X g d\eta=0,$$
for any continuous function on $X$, then $\eta$ is purely finitely additive measure.
\end{lemma}

\begin{proof}  According to the Definition 1.13 of \cite{Yos} we have to show that any countably additive measure
$\kappa$ satisfying 
 \begin{equation} \label{eq1}
0\le\kappa\le \eta
\end{equation}
 is a zero measure. Let $\kappa$ satisfy (\ref{eq1}). Then for any continuous function $g$, we have
\begin{equation*} 
0\le\kappa(g)\le \eta(g)=0.
\end{equation*}
Therefore $\kappa(g)=0$ for all continuous functions $g$. Since $\kappa$ is a countably additive measure,  $\kappa=0$.
\end{proof}

\textbf{Acknowledgments:} The authors are grateful to the anonymous reviewer for his very detailed comments which helped to improve the paper.

\end{document}